\documentclass[11pt]{article}
\usepackage{amsmath,amssymb}

\newcommand{\Var}{{\cal{V}_{\mathbf{k}}}}

\newcommand{\VarC}{{\cal{V}_{\mathbb{C}}}}

\newcommand{\St}{{\rm Stck}_{\mathbf{k}}}
\newcommand{\Zar}{{\cal{Z}}_{\mathbf{k}}}

\def\1{\underline{1}}

\def\AA{{\mathbb A}}
\def\P{\mathbb P}

\def\bL{{\mathbb L}}
\def\Z{{\mathbb Z}}
\def\Q{{\mathbb Q}}
\def\C{{\mathbb C}}

\def\Gr{\mbox{\bf Gr}}

\newtheorem{proposition}{Proposition}

\newenvironment{definition}
{\smallskip\noindent{\bf Definition\/}:}{\smallskip\par}
\newenvironment{corollary}
{\smallskip\noindent{\bf Corollary\/}.}{\smallskip\par}
\newenvironment{example}
{\smallskip\noindent{\bf Example\/}.}{\smallskip\par}

\newenvironment{remark}
{\smallskip\noindent{\bf Remark\/}.}{\smallskip\par}

\newenvironment{proof}
{\noindent{\bf Proof\/}.}{{ $\square$}\smallskip\par}

\title{On the pre-$\lambda$ ring structure on the Grothendieck ring of stacks and the power structures
 over it.
}
\author{S.M.~Gusein-Zade \thanks{Partially supported by the grants
RFBR-10-01-00678,
NSh-8462.2010.1 and by the research program 
of the Universidad Complutense-Grupo Santander.
Address: Moscow State University, Faculty
of Mathematics and Mechanics, Moscow, GSP-1, 119991, Russia. E-mail:
sabir\symbol{'100}mccme.ru} \and I.~Luengo \thanks{The last two authors are partially
supported by the grants MTM2010-21740-C02-01 and Grupo Singular CCG07-UCM/ESP-2695-921020. Address: University
Complutense de Madrid, Dept. of Algebra, Madrid, 28040, Spain.
E-mail: iluengo\symbol{'100}mat.ucm.es, amelle\symbol{'100}mat.ucm.es} \and
A.~Melle--Hern\'andez \thanks{Address: Instituto de Ciencias Matem\'aticas 
Complutense-CSIC-Aut\'onoma-Carlos III, Spain.
}}
\date{}
\begin{document}

\maketitle

\begin{abstract}
{\sloppy We discuss a generalization (``extension'') of the pre-$\lambda$ structure on the Gro\-then\-dieck ring of
quasi-projective varieties
and of the corresponding power structure over it to the Grothendieck ring of stacks, discuss
some of their properties and give some explicit formulae for the Kapranov zeta function for some stacks.}
\end{abstract}

\section{Introduction}
Let $K_0(\Var)$ be the Grothendieck ring of  quasi-projective
varieties over an algebraically closed field ${\mathbf{k}}$. This is the Abelian group generated by the classes $[X]$ of all 
quasi-projective ${\mathbf{k}}$-varieties $X$ modulo the relations:
\begin{enumerate}
\item[1)] $[X]=[Y]$ for isomorphic $X$ and $Y$;
\item[2)] $[X]=[Y]+[X\setminus Y]$ when $Y$ is a Zariski closed subvariety of $X$.
\end{enumerate}
The multiplication in $K_0(\Var)$ is defined by the Cartesian product
of varieties: $[X_1]\cdot [X_2]=[X_1\times X_2]$.
The class $[\AA^1_{\mathbf{k}}]\in K_0(\Var)$ of the  affine line is
denoted by $\bL$. 

In \cite{GLM04}, there was defined a notion of a \emph{power structure} over a ring and there was 
described a natural power structure over
the Grothendieck ring $K_0(\VarC)$ of complex quasi-projective
varieties. This means that, for a series $A(T)=1+a_1T+a_2T^2+\ldots \in 1+T \cdot K_0(\VarC)[[T]]$
and for an element $m\in K_0(\VarC)$,
one defines a series $(A(T))^m\in 1+ T \cdot K_0(\VarC)[[T]]$ 
so that all the usual properties of the exponential function hold.
Here we assume the base field ${\mathbf{k}}$ to be an arbitrary algebraically closed field
and therefore  we indicate changes required for that.

A special property of this power structure which is important for applications is its \emph{effectiveness}. This means that if all the coefficients $a_i$ of the series $A(T)$ 
and the exponent $m$ are represented by classes of quasi-projective varieties (i.e. $a_i=[A_i]$ and $m=[M]$),
then all the coefficients of the series $A(T)^m$ are represented by classes of quasi-projective varieties
too, not by classes of virtual varieties (see e.g. \cite{GLM06}, \cite{GLM07}). 
This property looks somewhat misterious. 
A power structure over a ring $R$ can be defined by a \emph{pre-$\lambda$ structure} on $R$. The 
pre-$\lambda$ structure on the Grothendieck ring $K_0(\VarC)$ which induces the power structure
described in \cite{GLM04} is defined by the Kapranov zeta-function introduced  in \cite{Kap}:
\begin{equation} \label{Kapranov}
\zeta_{X}(T):=1+[S^1 X]\cdot T+[S^2 X]\cdot T^2+[S^3 X]\cdot T^3+\ldots\,,
\end{equation}
where $S^k X:=X^k/S_k$ is the $k$-th-symmetric power of the variety $X$.
The use of another (natural)  pre-$\lambda$ structure on the Grothendieck ring $K_0(\VarC)$ (opposite to the one defined by (\ref{Kapranov})) 
induces a power structure over it which is not effective.

For a field ${\mathbf{k}}$ of positive characteristic the Kapranov zeta-function (\ref{Kapranov}) also
defines a pre-$\lambda$ structure on the Grothendieck ring $K_0(\Var)$ (see \cite[Proposition 1.1 i)]{Ek:09-2})
and therefore a power structure over it. However in \cite{Ek:09-2} it is explained that (over a field of positive
characteristic) one should consider the following, a little bit different,  pre-$\lambda$ structure on  $K_0(\Var)$.
 
Let $X^k=\coprod\limits_{\{k_i\}} X^k_{\{k_i\}}$ be the natural representation
of the Cartesian power $X^k$
of $X$ as the union of the strata corresponding to all the partitions of $k$.
The stratum corresponding to the partition  
$\{k_i\}:\ \sum i k_i=k$, consists of the $k$-tuples $(x_1,\ldots,x_k)$ with
(exactly) $k_i$ groups with (exactly) $i$ equal elements. 
For the partition $\{k_i\}$, let $\prod\limits_{i} X^{k_i}\setminus \Delta$
be the complement of the large diagonal
which consist of $\sum\limits_{i}k_i$-tuples of points of $X$ with at
least two coinciding ones. 
The group $\prod\limits_{i} S_{k_i}$ acts freely on
$\prod\limits_{i} X^{k_i}\setminus \Delta$.
One has $X^k_{\{k_i\}}/S_{k}\cong (\prod\limits_{i}
X^{k_i}\setminus \Delta)/\prod\limits_{i} S_{k_i}$.
One should use  $[S^k X]^*=\sum\limits_{\{k_i\}} 
[(\prod\limits_{i} X^{k_i})\setminus \Delta)/\prod\limits_{i} S_{k_i}]$
instead of $[S^k X]$ in the definition (\ref{Kapranov}).
For the base field ${\mathbf{k}}$ of characteristic zero
the classes  $[S^k X]$ and  $[S^k X]^*$ coincide. Moreover one can see that
the geometric definition of the power structure in \cite{GLM04}, being
literally extended to a field of an arbitrary characteristic, corresponds,
in fact, to this modification of the Kapranov zeta-function.
Indeed in 
\cite{GLM04}, the series $\left(A(T)\right)^{[X]}$ for $A(T)=1+\sum\limits_{i=1}^\infty  [A_i]\,T^i$ where $[A_i]$ are 
quasiprojective varieties and  for a quasi-projective variety 
$X$ is defined as
\begin{equation} \label{eq00}
\left(A(T)\right)^{[X]}
:=1+\sum_{k=1}^\infty
\left
\{\sum_{\{k_i\}:\sum ik_i=k}
\left[
\left((
(\prod_i X^{k_i}
)
\setminus\Delta)
\times\prod_i A_i^{k_i}\right)/\prod_iS_{k_i}
\right]
\right\}
\cdot T^k,
\end{equation}
where the group $S_{k_i}$  acts by permuting corresponding $k_i$ factors in
$\prod\limits_i X^{k_i}\supset (\prod_i X^{k_i})\setminus\Delta$
and the spaces $A_i$ simultaneously.
If all $A_i$ are points (i.e. $[A_i]=1$), the coeficient at $T^k$ in 
$(1+T+T^2+\ldots)^{[X]}=(1-T)^{-[X]}$ is just   
$\sum\limits_{\{k_i\}} 
[(\prod\limits_{i} X^{k_i}\setminus \Delta)/\prod\limits_{i} S_{k_i}]=[S^k X]^*$.

There are natural generalizations of the pre-$\lambda$ structure on the Grothendieck ring $K_0(\Var)$ 
to other situations, say to  the Grothendieck ring $K_0(\St)$ of algebraic stacks of finite type over ${\mathbf{k}}$
all of whose automorphism group schemes are affine. Though the corresponding 
pre-$\lambda$ structure on  the Grothendieck ring $K_0(\St)$ is effective (i.e. all the coeficients 
of the series $\zeta_{X}(T)$ are represented by effective stacks: see below) the corresponding 
power structure fails to be effective.

In this paper we describe a generalization (``extension'') of the pre-$\lambda$ structure on the ring $K_0(\Var)$
and of the corresponding power structure over it to the  Grothendieck ring $K_0(\St)$ of stacks, discuss
some of their properties and give some explicit formulae for the Kapranov zeta function for some stacks.

\section{Kapranov zeta function for some stacks}
A \emph{pre-$\lambda$ structure} on a ring $R$ is given by a series
$\lambda_a(T)\in 1+T\cdot R[[T]]$ defined for each $a\in R$ so that
\begin{enumerate}
\item[1)] $\lambda_a(T)=1+aT\, \mbox{ mod }T^2$,
\item[2)] $\lambda_{a+b}(T)=\lambda_a(T)\lambda_b(T)$ for $a,b\in R$.
\end{enumerate}

A natural 
pre-$\lambda$ structure on the Grothendieck ring $K_0(\Var)$ of quasi-projective varieties 
is defined by the Kapranov zeta-function:
$$
\zeta_{X}(T):=1+[S^1 X]\cdot T+[S^2 X]\cdot T^2+[S^3 X]\cdot T^3+\ldots\,
$$
(where $S^k X:=X^k/S_k$ is the $k$-th-symmetric power of the variety $X$) modified, for fields of positive characteristic, 
in the described above way.
The Kapranov zeta-function possesses the property $\zeta_{\bL^s X}(T)=\zeta_{X}(\bL^sT)$. 
In \cite[Statement 3]{GLM04}  this property was proved for ${\mathbf{k}}=\C$. 
The proof is based on Statement~2 therein, for any integer
$s\ge0$, $\left(A(\bL^s T)\right)^{[X]}
= \left(A(T)^{[X]}\right)\mbox{\raisebox{-0.5ex}{$\vert$}}{}_{T\mapsto {\bL^s T}}$.
One considers the corresponding parts
$$
V=\left((
(\prod_i X^{k_i}) \setminus \Delta)
\times\prod_i A_i^{k_i}\right)/\prod_iS_{k_i}
$$
and 
$$\widetilde V=\left((
(\prod_i X^{k_i}) \setminus \Delta)
\times\prod_i (\bL^{si}\times A_i)^{k_i}\right)/\prod_iS_{k_i}
$$
of the coefficients of the series $\left(A(T)\right)^{[X]}$ and
$\left(A(\bL^s T)\right)^{[X]}$.
The  natural map $\widetilde V \to V$ is a Zariski locally trivial vector bundle of rank $sk$.
In \cite{GLM04} this is deduced from \cite{Serre}. For the base field ${\mathbf{k}}$ of positive characteristic 
one can use \cite[Section 7, Proposition 7]{Mumf}. 

The property $\zeta_{\bL^s X}(T)=\zeta_{X}(\bL^sT)$  permits to ``extend'' the pre-$\lambda$ structure to the localization
$K_0(\Var)[\bL^{-1}]$ of the Grothendieck ring $K_0(\Var)$ in the multiplicative set $\{\bL^n\}$.
Thus for an element $M\in K_0(\Var)[\bL^{-1}]$ its Kapranov zeta function is defined as 
$\zeta_M(T):=\zeta_{\bL^s M}(\bL^{-s}T)$  for $s$ large enough so that  $\bL^s M\in K_0(\Var)$.
In particular 
$$
\zeta_{\bL^{-s}}(T)=1+ \bL^{-s} \cdot T+ \bL^{-2s}\cdot T^2+ \bL^{-3s}\cdot T^3+\ldots=
\frac{1}{1-\bL^{-s} T}\, .
$$
In the Grothendieck ring $K_0(\Var)$ one has 
$[GL(n)]=(\bL^n-1)(\bL^n-\bL)\dots(\bL^n-\bL^{n-1})$.

Let $K_0(\St)$ be the Grothendieck group  of algebraic stacks of finite type over ${\mathbf{k}}$
all of whose automorphism group schemes are affine with  relations:
\begin{enumerate}
\item[1)]  $[X]$ depends only on the isomorphism class of $X$;
\item[2)] $[X]=[Y]+[U]$ if $Y$ is a closed substack of $X$ and $U$ is its complement;
\item[3)] if $E\to X$ is a vector bundle of constant rank $n$ then $[E]=\bL^n [X]$, where $\bL$
also denotes the class $[\AA^1_{\mathbf{k}}]\in K_0(\St)$ of the affine line.
\end{enumerate}
The multiplication in $K_0(\St)$ is defined by the fibred product of stacks: $[X_1]\cdot [X_2]=[X_1\times X_2]$.
The following property holds in $K_0(\St)$, see \cite{Ek:09-1}. 
Let $\Zar$ be the class of connected group schemes for which torsors 
(principal $G$-bundles in the \'etale topology)  
are trivial in the Zariski topology (e.g. $GL(n)$ and $SL(n)$ are in $\Zar$). 
For a group $G\in \Zar$, $G$ is affine and connected. 
For any group $G\in \Zar$, $[G][BG]=1$ in $K_0(\St)$ where $BG$ is the classifying  stack of the group $G$ (e.g 
$[BGL(1)]=\frac{1}{\bL-1}$).
Thus, in particular, $\bL$ and $\bL^n-1$ for any $n\geq 1$ are invertible in $K_0(\St)$.

There is a natural ring homomorphism from $K_0(\Var)$ to $K_0(\St)$ 
by considering an algebraic variety as an algebraic stack.
In \cite[Theorem 1.2]{Ek:09-1}(see also \cite{To05}), it was shown that the natural ring homomorphism 
\begin{equation}\label{isom}
K_0(\Var)[\bL^{-1}][(\bL^n-1)^{-1}, \forall n\geq 1] \longrightarrow K_0(\St)
\end{equation}
is an \emph{isomorphism} between the Grothendieck ring $K_0(\St)$ 
of algebraic stacks with affine stabilizers and the Grothendieck ring $K_0(\Var)$ of 
varieties localized by the elements $\bL:=[\AA^1_{\mathbf{k}}]$ and $\bL^n-1$,  for all $n\geq 1$. 
(The fact that these rings tensored by the field $\Q$ of rational numbers are isomorphic 
was proved in \cite{Jo07}). 

In \cite{Ek:09-2}, it was shown that there is a  pre-$\lambda$ structure on the ring $K_0(\St)$ 
(extending the one on $K_0(\Var)[\bL^{-1}]$). The proof there is implicit. 
Here we shall give somewhat more explicit formulae for this pre-$\lambda$ structure
(i.e. for the corresponding generalization of 
the Kapranov zeta function) and compute it for some examples. In particular we compute 
``symmetric powers'' of the classifying stack $BGL(1)$ of the group $GL(1)={\mathbf{k}}^*$.

Let $R_k(q_1,\dots,q_k)$ be the rational function in $q_1,\ldots,q_k$ defined by its Taylor expansion 
\begin{equation}\label{rat-funct}
 R_k(q_1,\dots,q_k):=\sum_{\underline{i}=(i_1,\ldots,i_k)\in \Z^k_{\geq 0}:\,\,
 i_s\ne i_r \,\,\textmd{for} \,\,s\ne r}  q_1^{i_1}  q_2^{i_2} \dots  q_k^{i_k} 
\end{equation} 

\noindent {\bf Fact 1.} One can see that 
\begin{equation*}
 R_k(q_1,\dots,q_k)=\sum_{\sigma\in S_k}  
\frac
{q_{\sigma(1)}^{k-1} q_{\sigma(2)}^{k-2}\dots
q_{\sigma(k-1)}}
{(1-q_{\sigma(1)})(1-q_{\sigma(1)}q_{\sigma(2)})\dots(1-q_{\sigma(1)}\dots q_{\sigma(k)})}.
\end{equation*} 
Here $\frac{q_{1}^{k-1} q_2^{k-2}\dots
q_{{k-1}}}{(1-q_{1})(1-q_{1}q_{2})\dots(1-q_{1}\dots q_{k})}$ is the part of the sum in (\ref{rat-funct}) over 
$\underline{i}=(i_1,\ldots,i_k)$ with $i_1>i_2>\ldots>i_k\geq 0$.

\noindent {\bf Fact 2.} For $\sum\limits_{j=1}^s n_j=k$, define
\begin{equation*}
R_{n_1,n_2,\ldots,n_s}(q_1,\ldots q_s):=\sum_{\{i_r^{(j)}\}}  q_1^{i_1^{(1)}} \dots q_1^{i_{n_1}^{(1)}}  
q_2^{i_1^{(2)}} \dots q_2^{i_{n_2}^{(2)}} \dots q_s^{i_1^{(s)}} \dots q_s^{i_{n_s}^{(s)}} ,
\end{equation*} 
where the sum is over all collections of $\{i_r^{(j)}\}$  with $1\leq j\leq s$, 
$1\leq r \leq n_j$, $i_1^{(j)}< i_2^{(j)} < \ldots < i_{n_j}^{(j)}$,
$i_{r_1}^{(j_1)}\neq  i_{r_2}^{(j_2)} $ for $(j_1,r_1)\ne (j_2,r_2)$. One can see that  
\begin{equation*}
R_{n_1,n_2,\ldots,n_s}(q_1,\ldots q_s)=R_k(\underbrace{q_1,\ldots,q_1}_{n_1},
\underbrace{q_2,\ldots,q_2}_{n_2},\ldots, \underbrace{q_s,\ldots,q_s}_{n_s})/\prod_{j=1}^s n_j ! .
\end{equation*} 

\medskip

Because of  the isomorphism (\ref{isom}), an element of the Grothendieck ring $K_0(\St)$ of stacks is of the form 
$a=\frac{M\bL^{-m}}{\prod_{i=1}^{\ell}(1-\bL^{-n_i})}$, where  $M \in K_0(\Var)$,
$m\in \Z$, $n_i\in \Z\setminus \{0\}$. (Without loss of generality one may assume that  
$m\in \Z_{\geq 0}$, $n_i\in \Z_{>0}$, however, for the formulae bellow this is not essential. ) 

In \cite[Lemma 2.2]{Ek:09-2}, the corresponding series $ \zeta_{a}(T)$, or rather 
$ (\zeta_{a}(-T))^{-1}(=\zeta_{-a}(-T))$, 
was described by a functional equation. A closed formula for   $ \zeta_{a}(T)$ is somewhat  involved. 
Therefore we shall give a formula for the Kapranov zeta function of an element $a$ of the form $ 
\frac{b\,\bL^{-m}}{1-\bL^{-n}}$, where $b\in K_0(\St)$ (in terms of the series  $ \zeta_{b}(T)$).
Let us denote the coefficients of the series $ \zeta_{b}(T)$ by $\sigma^k b\in K_0(\St):$
$$
\zeta_{b}(T):=1+\sum_{k=1}^\infty  \sigma^k b \,\,\cdot  T^k,
$$
$\sigma^1 b=b$. For $b=[X]$, where $X$ is quasi-projective variety, $\sigma ^k b$ is the class $[S^k X]$ 
of its symmetric power $S^k X$, (modified, for a field ${\mathbf{k}}$ of positive characteristic, in the described way).

\begin{remark} 
Note that, for  $\sum\limits_{j=1}^s n_j=k$ and for arbitrary $\ell_1,\ldots, \ell_s$, the corresponding element 
$R_{n_1,n_2,\ldots,n_s}((\bL^{-1})^{\ell_1},\ldots,(\bL^{-1})^{\ell_s})\in K_0(\St)$ 
is effective, in the sense that it is equal to the class of a quasiprojective variety divided by 
$\bL^r\prod_{i} [GL(r_i)]$. This follows directly form the formula for $R_k(q_1,\ldots,q_k)$.
\end{remark}

\begin{proposition}
For   an element $a=\frac{b q^{m}}{1-q^{n}}$, where $b\in K_0(\St)$ and $q=\bL^{-1}$, one has 
$$
\zeta_{a}(T)=1+\sum_{k=1}^\infty  \left(\sum_{\underline{k}:\sum\limits_{j=1}^s j k_j= k} 
R_{k_1,2 k_2,\ldots,s k_s}(q^{n},q^{2n},\ldots, q^{s n})  \prod_{j} (\sigma^j b)^{k_j}\right) 
 q^{km}  T^k,
$$
where the second sum runs over all partitions 
$\underline{k}=(k_1,\ldots,k_s): \sum\limits_{j=1}^s j k_j= k$ of the integer $k$.
\end{proposition}

\begin{corollary}
 For an element $a=\frac{1}{1-q^{n}}\in K_0(\St)$ and $q=\bL^{-1}$, one has 
$$
\zeta_{\frac{1}{1-q^{n}}}(T)=1+\sum_{k=1}^\infty  \left(\sum_{\underline{k}:\sum\limits_{j=1}^s j k_j= k} 
R_{k_1,2 k_2,\ldots,s k_s}(q^{n},q^{2n},\ldots, q^{s n})\right)  T^k,
$$
where the second sum runs over all partitions 
$\underline{k}=(k_1,\ldots,k_s): \sum\limits_{j=1}^s j k_j= k$ of $k$.
\end{corollary}

\begin{proof}
 According to \cite{Ek:09-2},
 the series $\zeta_{a}(T)$ is defined by the following functional equation $\zeta_{a}(T)/\zeta_{a}(q^n T)= 
\zeta_b(q^m T)$ and the coefficients of this series are polynomials in 
$\sigma^i b$, $q^{\pm 1}$, and $(1-q^j)^{-1}$ for all $j>0$. Polynomials in 
$q^{\pm 1}$ and $(1-q^j)^{-1}$ can be defined by their Laurent expansion at $q=0$.
Consider the series in $T$, whose coefficients are elements  of $\Z[q^{\pm 1}, \sigma^i b]][[q]]$, of the form 
$$
\psi(T):=\prod_{i=0}^\infty \zeta_b(q^{m+in} T).
$$ 
One can see that 
$$
\psi(T)/\psi(q^nT)=
\prod_{i=0}^\infty \zeta_b(q^{m+in} T)(\prod_{i=0}^\infty \zeta_b(q^{m+(i+1)n} T))^{-1}=\zeta_b(q^m T).
$$
Also one can show (see bellow) that the coefficients of the series $\psi(T)$ are Laurent expansions in $q$ 
of polynomials  in $\sigma^k b$, $q^{\pm 1}$ and $(1-q^j)^{-1}$. Therefore  $\psi(T)=\zeta_{a}(T)$.

Thus one has 
$$ 
\zeta_{a}(T)=\prod_{i=1}^\infty (1+q^{m+in} \sigma^1 b \, T+q^{2(m+in)} \sigma^2 b \, T^2+\ldots).
$$
The coefficient of $T^k$ in this  series is the sum over all partitions $\underline{k}=(k_1,\ldots,k_s)$ of $k$
with $\sum\limits_{j=1}^s j k_j=k$ of the series 
\begin{equation*}
\prod_{j} (\sigma^j b)^{k_j} \sum_{\{i_r^{(j)}\}} q^{m+i_1^{(1)}n}\dots q^{m+i_{k_1}^{(1)}n} 
q^{m+2i_1^{(2)}n}\dots q^{m+2 i_{k_2}^{(2)}n} \ldots \ldots
q^{m+si_1^{(2)}n}\dots q^{m+s i_{k_s}^{(s)}n} 
\end{equation*}
where the sum  is over all collections of $\{i_r^{(j)}\}$  with $1\leq j\leq s$, 
$1\leq r \leq n_j$, $i_1^{(j)}< i_2^{(j)} < \ldots < i_{n_j}^{(j)}$,
$i_{r_1}^{(j_1)}\neq  i_{r_2}^{(j_2)} $ for $(j_1,r_1)\ne (j_2,r_2)$. This implies the statement.
\end{proof}

\begin{corollary}
The Kapranov zeta function for stacks is effective in the sense that, for a stack $X$, the coefficients of
$\zeta_{[X]}(T)$ are represented by classes of stacks in $K_0(\St)$.
\end{corollary}

This was shown by T. Ekedahl who identified the coefficients of $\zeta_{[X]}(T)$ with the classes of the
symmetric products of the stack $X$; \cite{Ek:09-2}.

\begin{example}
The Kapranov zeta function of the element $a=\frac{q^m}{1-q^{n}}\in K_0(\St)$ is equal to 
\begin{equation}\label{zeta-some}
\zeta_{\frac{q^m}{1-q^{n}}}(T)=1+\sum_{k=1}^\infty  \frac{q^{mk}}{(1-q^n)(1-q^{2n})\ldots(1-q^{kn})} T^k.
\end{equation}
In particular, 
\begin{equation}\label{zeta-gl1}
\zeta_{[BGL(1)]}(T)=\zeta_{\frac{1}{\bL-1}}(T)=1+\sum_{k=1}^\infty \frac{\bL^{k^2-k}}{(\bL^k-\bL^{k-1})
\ldots(\bL^k-1)} T^k.
\end{equation}
Thus 
\begin{equation}\label{sym-gl1}
\sigma^k[BGL(1)]= \bL^{k^2-k} [BGL(k)]. 
\end{equation}
As in the proof of the proposition with $b=1$ consider  
$$\psi(T)=1+\sum_{k=1}^\infty  \frac{q^{mk}}{(1-q^n)(1-q^{2n})\ldots(1-q^{kn})} T^k$$
and we have to show that the following functional equation $\psi(T)/\psi(q^nT)=\zeta_1(q^m T)$ holds. 
Thus the equation (\ref{zeta-some}) follows from the equalities 
$$
\left(1+\sum_{k=1}^\infty  \frac{q^{(m+n)k}}{(1-q^n)(1-q^{2n})\ldots(1-q^{kn})} T^k\right)
\left( \sum_{k=0}^\infty q^{mk} T^k \right)=
$$
$$
=1+\sum_{k=1}^\infty \sum_{j=0}^k \frac{q^{(m+n)j+m(k-j)}}{(1-q^n)(1-q^{2n})\ldots(1-q^{jn})} T^k =
$$
$$
=
1+\sum_{k=1}^\infty \sum_{j=0}^k \frac{q^{nj}}{(1-q^n)(1-q^{2n})\ldots(1-q^{jn})} q^{mk}T^k= 
1+\sum_{k=1}^\infty  \frac{q^{mk}}{(1-q^n)(1-q^{2n})\ldots(1-q^{kn})} T^k
$$

\end{example}

\begin{remark}
Notice that the classifying space for the group $GL(1)=\C^*$ 
in the topological sense is $\C\P^\infty$, whence the topological 
classifing space for the group $GL(k)$ is the (infinite dimensional) Grassmannian 
$\Gr(k,\infty)$. If one defines $
\zeta_{\C\P^\infty}(T)$ as 
$$
\zeta_{\C\P^\infty}(T):=\lim_{N \to \infty} \zeta_{\C\P^N}(T)=\prod_{j=0}^\infty \zeta_{\bL^j}(T)
$$
(whatever this means) and  $[\Gr(k,\infty)]$ as a series in $\bL$ equal to  $[\Gr(k,\infty)]=
\lim_{N \to \infty}  [\Gr(k,N)]$ one gets 
\begin{equation}\label{cpinfinity}
\zeta_{\C\P^\infty}(T)=1+\sum_{k=1}^\infty [\Gr(k,\infty)] T^k ,
\end{equation}
which is similar to (\ref{sym-gl1}) up to a dimensional factor $\bL$ in some power.
The equation (\ref{cpinfinity}) means that $[S^k \C\P^\infty]=[\Gr(k,\infty)]$. Moreover 
$\C\P^\infty$ and $\Gr(k,\infty)$ have decompositons into quasi-projective varieties (compatible with the inclusions 
$S^k \C\P^N \subset S^k \C\P^\infty$  and $\Gr(k,N)\subset \Gr(k,\infty)$) such that components of $S^k \C\P^\infty$
and of $\Gr(k,\infty)$ are pair-wise isomorphic. This uses the fact  that $S^k \C^n$ and $\C^{kn}$ have decompositions 
into quasi-projective varieties such that their components are pair-wise isomorphic.
See a proof of these statements in \cite{piecewise}.
\end{remark}
 
\section{A power structure over $ K_0(\St)$}
\begin{definition}
A {\em power structure} over a (semi)ring $R$ with a unit is a map
$\left(1+T\cdot R[[T]]\right)\times {R} \to 1+T\cdot R[[T]]$:
$(A(T),m)\mapsto \left(A(T)\right)^{m}$,
which possesses the following properties:
\begin{enumerate}
\item[1)] $\left(A(T)\right)^0=1$,
\item[2)] $\left(A(T)\right)^1=A(T)$,
\item[3)] $\left(A(T)\cdot B(T)\right)^{m}=\left(A(T)\right)^{m}\cdot
\left(B(T)\right)^{m}$,
\item[4)] $\left(A(T)\right)^{m+n}=\left(A(T)\right)^{m}\cdot
\left(A(T)\right)^{n}$,
\item[5)] $\left(A(T)\right)^{mn}=\left(\left(A(T)\right)^{n}\right)^{m}$,
\item[6)] $(1+T)^m=1+mT+$ terms of higher degree,
\item[7)] $\left(A(T^k)\right)^m =
\left(A(T)\right)^m\raisebox{-0.5ex}{$\vert$}{}_{T\mapsto T^k}$, $k\geq 1$.
\end{enumerate}
\end{definition}

As it was explained in \cite{GLM06},
a pre-$\lambda$ structure $\lambda_a(T)\in 1+T\cdot R[[T]]$ on a ring $R$ defines a power structure over it. 
To define the series 
$(A(T))^m$ for $A(T)=1+a_1T+a_2T^2+\ldots$, $a_i\in R$,   $m\in R$ one has to represent in a unique way  
$A(T)$ as a product of the form
$A(T)=\prod\limits_{k=1}^\infty \lambda_{b_k}(T^k)$, with $b_i\in R$
and then
\begin{equation}\label{eq1}
\left(A(T)\right)^m:=\prod\limits_{k=1}^\infty \lambda_{mb_k}(T^k).
\end{equation}

To each pre-$\lambda$ structure $\lambda_a(T)$ on a ring $R$, there corresponds 
the so called \emph{opposite} pre-$\lambda$ structure $\lambda^{op}_a(T):=(\lambda_a(-T))^{-1}$.
Notice that the power structures defined by these two pre-$\lambda$-structures are in general different.

One can easily see this fact for the pre-$\lambda$ structure on the polynomial ring
$\Z[u_1,\ldots,u_k]$ described in \cite{GLM06}.

One can say that there are at least 4 ``natural'' pre-$\lambda$ structures
over the Grothendieck ring $K_0(\Var)$ of quasi-projective varieties. They are 
those defined by the Kapranov zeta function 
$$
\zeta_{M}(T):=1+[S^1 M]\cdot T+[S^2 X]\cdot T^2+[S^3 M]\cdot T^3+\ldots\,,
$$
and by the generating series of the configuration space $M_k=(M^k\setminus \Delta)/S_k$ 
of unordered $k$-tuples of diferent points of a quasi-projevtive variety  $M$:
$$
\varphi_M(T):=1+[M]\cdot T+[M_2]\cdot T^2+[M_3]\cdot T^3+\ldots\,,
$$
and their corresponding opposites.

One can see that the power structure corresponding to the first two
pre-$\lambda$ structures coincide (this is a consequence of the equation 
 $\varphi_M(T)=(1+T)^M$ see Example in \cite{GLM04}, where in this equality 
the power estructure is defined by  
$\zeta_{M}(T)$), whence those corresponding to their 
opposites are different from this one (and coincide with each other).
The advantage of the power structure defined by 
the Kapranov zeta function (or by the series  $\varphi_M(T)$) is the fact that it is defined 
over the Grothendieck semiring of quasi-projective
varieties (whose elements are represented by ``genuine'' varieties not by virtual ones). We shall 
say that this power structure is \emph{effective}.  
This follows from the mentioned above geometric description of the coefficient of the series 
$(A(T))^M$ for $A(T)=1+[A_1]\cdot T+[A_2]\cdot T^2+[A_3]\cdot T^3+\ldots\,,$
and $A_k$ and $M$ are quasiprojective varieties.
We do not know how one can prove effectivness of this power structure from 
the definition via $\varphi_M(T)$. With their opposite power structure one has 
 
\noindent {\bf Statement 1.} The power structure over the Grothendieck ring $K_0(\Var)$ 
defined by the pre-$\lambda$ structure $\lambda_X(T):=(\zeta_{X}(-T))^{-1}$ is not effective.

Fon a class $[M]$ of a quasi-projective variety $M$,   $\lambda_{[M]}(T)$ is the inverse of the series
$(1-[M]T+[S^2M]T^2-[S^3M]T+\ldots)$, that is  $\lambda_{[M]}(T)=1+[M]T+([M]^2-[S^2M])T^2+\ldots$.
If $M$  is smooth elliptic curve $C$, then the Hodge-Deligne polynomial $e([C]^2-[S^2C])$ is equal to
$-u^2v-uv^2+u^2+v^2+2uv-u-v$, 
since $e([C])=1-u-v+uv$ and $e([S^2C])$ is the coefficient at $T^2$ in the series 
$\frac{1}{1-T}(1-uT)(1-vT)\frac{1}{1+uvT}$. Such polynomial cannot be the 
Hodge-Deligne polynomial of a quasi-projective variety since its 
homogeneous part of highest degree is not of the form $\ell (uv)^n$, $\ell$ a non-negative  integer. 
 
\noindent {\bf Statement 2.} The power structure over the Grothendieck ring $K_0(\St)$ 
defined by the the Kapranov zeta function is not effective.

We shall show that the series $(1+T)^{[BGL(1)]}=(1+T)^{\frac{1}{\bL-1}}$ is not effective,
in the sense that it contains non-effective coeficientes. Namely we shall show that the second coefficient
of such a series in not effective. One has: 
$$
(1+T)^{\frac{1}{\bL-1}}=\frac{(1-T)^{-\frac{1}{\bL-1}}}{(1-T^2)^{-\frac{1}{\bL-1}}}=
\frac{\zeta_{\frac{1}{\bL-1}}(T)}{\zeta_{\frac{1}{\bL-1}}(T^2)}\,.
$$ 
According to (\ref{zeta-gl1}) 
$$
(1+T)^{\frac{1}{\bL-1}} \mod T^3=\left(1+\frac{1}{\bL-1} T+ \frac{\bL^2}{(\bL^2-1)(\bL^2-\bL)} T^2\right)
\left(1-\frac{1}{\bL-1} T^2\right) \mod T^3 .
$$
The coefficient at $T^2$ is equal to 
$$
\frac{\bL^2}{(\bL^2-1)(\bL^2-\bL)}-\frac{1}{\bL-1}=\frac{-\bL^3+\bL^2+\bL}{[GL(2)]}\, .
$$

It is not effective since the term of highest degreee in the Hodge-Deligne polynomial of the 
numerator has negative coefficient. 

For a quasi-projective variety $X$, the coefficient at $T^n$ in the series $(1+T)^{[X]}$
is represented by the configuration space $X^n\setminus\Delta/S^n$ of $n$ different points in $X$.
Statement 2 gives a hint for the conjecture that the notion of the configuration spaces of different points
cannot be defined for stacks (at least in a form close to that for varieties).

\end{document}